\documentclass[11pt]{article}
\usepackage[latin1]{inputenc}
\usepackage{epsfig}
\usepackage{color}
\usepackage[british,english]{babel}
\usepackage{amsthm}
\usepackage{amsmath}
\usepackage{amsfonts}
\usepackage{amssymb}
\usepackage{graphicx}
\newtheorem{corollary}{Corollary}[section]
\newtheorem{definition}[corollary]{Definition}

\newtheorem{lemma}[corollary]{Lemma}
\newtheorem{proposition}[corollary]{Proposition}
\newtheorem{remark}[corollary]{Remark}
\newtheorem{theorem}[corollary]{Theorem}
\newfont{\sBlackboard}{msbm10 scaled 900}

\newcommand{\mylabel}[1]{\label{#1}
            \ifx\undefined\stillediting
            \else \fbox{$#1$}\fi }
\newcommand{\BE}{\begin{equation}}

\newcommand{\EEQ}{\end{equation}}
\newcommand{\rfb}[1]{\mbox{\rm
   (\ref{#1})}\ifx\undefined\stillediting\else:\fbox{$#1$}\fi}

\newfont{\Blackboard}{msbm10 scaled 1200}

\newfont{\roma}{cmr10 scaled 1200}

\def\CC{\rm \hbox{C\kern-.56em\raise.4ex
         \hbox{$\scriptscriptstyle |$}\kern+0.5 em }}
\newcommand{\be}{\begin{equation}}
\newcommand{\ee}{\end{equation}}
\newcommand{\beq}{\begin{eqnarray}}
\newcommand{\eeq}{\end{eqnarray}}
\newcommand{\beqs}{\begin{eqnarray*}}
\newcommand{\eeqs}{\end{eqnarray*}}
\newcommand{\bt}{\begin{Theorem}}
\newcommand{\et}{\end{Theorem}}
\newcommand{\br}{\begin{remark}}
\newcommand{\er}{\end{remark}}
\newcommand{\bc}{\begin{Corollary}}
\newcommand{\ec}{\end{Corollary}}
\newcommand{\el}{\end{Lemma}}
\newcommand{\bd}{\begin{definition}}
\newcommand{\ed}{\end{definition}}

\newcommand{\C}{{\mathbb C}}

\newcommand{\nline}  {{\mathbb N}}
\newcommand{\rline}  {{\mathbb R}}

\def\cA{{\cal A}}
\def\A{{\cal A}}

\def\cH{{\cal H}}

\def\cL{{\cal L}}

\def\h{{\mathcal H}}
\newcommand{\mm}    {{\hbox{\hskip 0.5pt}}}

\newcommand{\bluff} {{\hbox{\raise 15pt \hbox{\mm}}}}
\makeatletter
\def\section{\@startsection {section}{1}{\z@}{-3.5ex plus -1ex minus
    -.2ex}{2.3ex plus .2ex}{\large\bf}}
\makeatother
\def\be{\begin{equation}}
\def\ee{\end{equation}}

\begin{document}

\thispagestyle{empty}
\title{\bf Stability of an abstract--wave equation with delay and a Kelvin--Voigt damping}
\author{Ka\"{\i}s AMMARI
\thanks{UR Analysis and Control of Pde, UR 13ES64, Department of Mathematics,
Faculty of Sciences of Monastir, University of Monastir, 5019 Monastir, Tunisia,
e-mail : kais.ammari@fsm.rnu.tn}
\, , Serge NICAISE \thanks{Universit\'e de Valenciennes et du Hainaut Cambr\'esis, LAMAV, FR CNRS 2956,  59313 Valenciennes Cedex 9, France, e-mail: snicaise@univ-valenciennes.fr}
\, and \, Cristina PIGNOTTI \thanks{Dipartimento di Ingegneria e Scienze dell'Informazione e Matematica, Universit\`a di L'Aquila, Via Vetoio, Loc. Coppito, 67010 L'Aquila, Italy, \,
e-mail : pignotti@univaq.it}}
\date{}
\maketitle
{\bf Abstract.} {\small
In this paper we consider a stabilization problem for an abstract wave equation with delay and a Kelvin--Voigt damping.
We prove an exponential stability result for appropriate damping coefficients. The proof of the main result is based on a
frequency--domain approach.} \\


\section{Introduction}
\setcounter{equation}{0}
Our main goal is to study the internal stabilization of a delayed abstract wave equation with a Kelvin--Voigt damping. More precisely, given a constant time delay $\tau >0,$ we consider the
system  given by:
\begin{align}
{}&u^{\prime \prime}(t) + a \, BB^* u^\prime (t)  + BB^*u(t - \tau) = 0 ,  && \mbox{\rm in} \quad (0, + \infty),\label{a1}\\
{}&u(0) = u_0,  \quad  u^\prime(0)=u_1, &&
\label{a4}
\\
{}&B^*u(t-\tau) = f_0(t-\tau),  \quad  && \mbox{\rm in} \quad (0,\tau),\label{a5}
\end{align}
where $a>0$ is a constant, $B : { D} (B) \subset H_1 \rightarrow H$ is a linear unbounded operator
from a Hilbert space $H_1$ into another Hilbert space $H$ equipped with the respective norms $||\cdot||_{H_1}$, $||\cdot||_H$ and   inner products $(\cdot,\cdot)_{H_1}$, $(\cdot,\cdot)_H$,
and $B^* : { D}(B^*) \subset H \rightarrow H_1$ is the adjoint of $B$.
The initial datum $(u_0, u_1, f_0)$ belongs to a suitable space.

We suppose that the operator $B^*$ satisfies the following coercivity assumption: there exists $C > 0$
such that
\be
\label{cors}
\left\|B^* v\right\|_{H_1} \geq C \, \left\|v\right\|_H, \, \forall \, v \in { D}(B^*).
\ee
For shortness we set $V={ D}(B^*)$ and we assume that it is closed with the norm
$\|v\|_V:=\left\|B^* v\right\|_{H_1}$ and that it is compactly embedded into $H$.

Delay effects arise in many applications and practical problems and
it is well--known that an arbitrarily small delay may destroy the well--posedness of the problem
\cite{jordanetall,dreheretall,pruss:93,racke}
or
 destabilize a
system which is uniformly asymptotically stable in absence of delay
(see e.g. \cite{Datko,Datko97},  \cite{NPSicon}, \cite{racke}).
Different strategies were recently developed to restitute either the well--posedness
or the stability. In the first case, one idea is to add a non--delay term, see  \cite{batkaibook,pruss:93}
for the heat equation.
In the second case,
we    refer to \cite{amman,ANP,DLP,NPSicon,NVCOCV10} for stability results for systems with time delay
where a  standard feedback compensating the destabilizing delay effect is introduced.
Nevertheless recent papers reveal that particular choices of the delay may restitute exponential stability property, see \cite{Gugat, ANP1}.

Note that the above system is exponentially stable in absence of
time delay, and if $a > 0$. On the other hand if $a=0$
and $-B B^*$ corresponds to the Laplace operator with Dirichlet boundary conditions
in a bounded domain of $\rline^n$, problem \eqref{a1}--\eqref{a5} is not well--posed, see \cite{jordanetall,dreheretall,pruss:93,racke}.
Therefore
in this paper in order to restitute the well-posedness character and its stability
we propose to add the Kelvin--Voigt damping term $ a \, BB^* u^\prime$.
Hence the stabilization of problem \eqref{a1}--\eqref{a5} is performed using a frequency domain approach
 combined with a precise spectral analysis.

The paper is organized as follows. The second section deals with the
well--posedness of the problem while, in the third section,  we perform the spectral analysis
of the associated operator. In section \ref{sectasanalysis},
we prove
the exponential stability of the  system \rfb{a1}--\rfb{a5} if $\tau\leq a$. In the last section we give an example of an application.

\section{Existence results}
\setcounter{equation}{0}

In this section we will give a
well--posedness result for problem \eqref{a1}--\eqref{a5} by
using semigroup theory.

Inspired from \cite{NPSicon}, we introduce the auxiliary variable
\begin{equation}\label{defz0}
z(\rho ,t)= B^* u(t-\tau\rho ),\quad \ \rho\in (0,1), \ t>0.
\end{equation}
Then, problem \eqref{a1}--\eqref{a5}
is equivalent to
\begin{align}
{}&u^{\prime \prime} (t) + a \, BB^*u^\prime (t) + Bz(1,t) = 0 ,  && \mbox{\rm in} \quad (0, + \infty),\label{a1bis}\\
{}& \tau z_t(\rho,t)+z_{\rho}(\rho, t)=0\quad &&\mbox{\rm in}\quad (0,1)\times (0,+\infty ),\label{Pz}\\
{}&u(0) = u_0,  \quad  u^\prime(0)=u_1, &&  \quad \label{a4bis}
\\
{}&z(\rho, 0) = f_0(- \rho \tau),  \quad  && \mbox{\rm in} \quad (0,1),\label{a5bis}
\\
{}&z(0,t)=B^*u(t),\quad &&    \ t>0.\label{a6}
\end{align}
If we denote
$$
U:=\left (
u,
u^\prime,
z
\right )^\top,
$$
then
$$
U^{\prime}:=\left (
u^\prime,
u^{\prime \prime},
z_t
\right )^\top=
\left (
u^\prime,
- a BB^* u^\prime - Bz(1,t),
-\tau^{-1}z_{\rho}
\right )^\top.
$$
Therefore, problem (\ref{a1bis})--(\ref{a6}) can be rewritten as
\begin{equation}\label{formulA}
\left\{
\begin{array}{l}
U^{\prime}={\cal A} U,\\
U(0)=\left (
u_0,
u_1,
f_0(-\cdot\tau )
\right )^\top,
\end{array}
\right.
\end{equation}
where the operator ${\cal A}$ is defined by
$$
{\cal A} \left (
\begin{array}{l}
u\\
v\\
z
\end{array}
\right ):=
\left (
\begin{array}{l}
v\\
- a BB^* v - Bz(\cdot,1)\\
-\tau^{-1}z_{\rho}
\end{array}
\right ),
$$
\noindent with domain
\begin{equation}\label{domain}
\begin{array}{l}
\displaystyle{
{ D} ({\cal A}):=\Big \{\ (u,v,z)^\top\in
{ D}(B^*)
\times { D}(B^*) \times H^1(0,1; H_1)\ : a B^* v + z(1)\in { D}(B),}\\
\hspace{2.5cm} \displaystyle{{B^*u=z(0)}
 \Big\},}
\end{array}
\end{equation}
in the Hilbert space
\begin{equation}\label{space}
{\cal H}:= { D}(B^*) \times H \times L^2(0,1;H_1),
\end{equation}
equipped with the standard inner product
\[
((u,v,z), (u_1,  v_1,  z_1))_{\cal H}
=
(B^*u,B^*u_{1})_{H_1} + (v,v_1)_H \, + \,
\xi \int_0^1 (z,z_1)_{H_1} \, d\rho,
\]
where $\xi>0$ is a parameter fixed later on.

We will show that ${\cal A}$ generates a $C_0$ semigroup on ${\cal H}$
by proving that ${\cal A}- c Id$ is maximal dissipative for an appropriate choice of $c$ in function of $\xi, \tau$ and $a$.
Namely we prove the next result.
\begin{lemma}\label{lmaxdiss}
If $\xi  > \frac{2\tau}{a}$, then there exists $a^*>0$ such that ${\cal A}-a_*^{-1} Id$ is maximal dissipative in $\cal H$.
\end{lemma}

\begin{proof}
Take $U=(u,v,z)^T\in D({\cal A}).$ Then we have
\beqs
({\cal A}(u,v,z), (u,  v,  z))_{\cal H}
=
(B^*v,B^*u)_{H_1} - ((B(aB^*v+z(1)),v)_H
\\
-\xi \tau^{-1} \int_0^1 (z_\rho,z)_{H_1}\, d\rho.
\eeqs
Hence, we get
\beqs
({\cal A}(u,v,z), (u,  v,  z))_{\cal H}
=
(B^*v,B^*u)_{H_1} -(aB^*v+z(1),B^*v)_{H_1}
\\
-\frac{\xi}{2 \tau}\left\|z(1)\right\|_{H_1}^2\,
+\frac{\xi}{2 \tau}\|z(0)\|_{H_1}^2.
\eeqs
Hence reminding that $z(0)=B^*u$ and using Young's inequality we find that
$$
\begin{array}{l}
\Re ({\cal A}(u,v,z), (u,  v,  z))_{\cal H}
\\\hspace{1 cm}
\leq (\varepsilon-a)
\|B^*v\|^2_{H_1} +(\frac{1}{2\varepsilon}-\frac{\xi}{2 \tau})\|z(1)\|^2_{H_1}
+(\frac{1}{2\varepsilon}
+\frac{\xi}{2 \tau})\|B^*u\|^2_{H_1}.
\end{array}
$$
Chosing $\varepsilon= \frac{a}{2}$, we find that
\beqs
\Re ({\cal A}(u,v,z), (u,  v,  z))_{\cal H}
\leq  - \frac{a}{2} \,
\|B^*v\|^2_{H_1} +  (\frac{1}{a}-\frac{\xi}{2 \tau})\|z(1)\|_{H_1}^2
+(\frac{1}{a}
+\frac{\xi}{2 \tau})\|B^*u\|^2_{H_1}.
\eeqs
The choice of $\xi$ is equivalent to  $\frac{1}{a}-\frac{\xi}{2 \tau} <0$,
and therefore for $a_* = \left(\frac{1}{a} + \frac{\xi}{2\tau} \right)^{-1}$,
\be
\label{dissp}
\Re ({\cal A}(u,v,z), (u,  v,  z))_{\cal H}
\leq - \frac{a}{2} \,
\|B^*v\|^2_{H_1} + (\frac{1}{a} - \frac{\xi}{2 \tau})\|z(1)\|_{H_1}^2 + a_*^{-1} \, \|B^*u\|^2_{H_1}.
\ee
As $\|B^*u\|^2_{H_1}\leq \|(u,v,z)\|^2_{\cal{H}}$, we get
$$
\Re (({\cal A} - a_*^{-1} Id)(u,v,z), (u,  v,  z))_{\cal H}
\leq - \frac{a}{2} \,
\|B^*v\|^2_{H_1} + (\frac{1}{a} - \frac{\xi}{2 \tau})\|z(1)\|_{H_1}^2 \leq 0,
$$
which directly leads to the dissipativeness of ${\cal A} - a_*^{-1} Id$.

Let us go on with the maximality, namely let us  show that $\lambda I -{\cal A}$ is surjective for a fixed $\lambda >0.$
Given $(f,g,h)^T\in {\cal H},$ we look for a   solution $U=(u,v,z)^T\in D({\cal A}) $ of
\be\label{surj}
(\lambda I- {\cal A}) \left (
\begin{array}{l}
u\\
v\\
z
\end{array}
\right )=
 \left (
\begin{array}{l}
f\\
g\\
h
\end{array}
\right ),
\ee
that is, verifying
\begin{equation}\label{max}
\left\{
\begin{array}{l}
\lambda u-v=f,\\
\lambda v + B(aB^*v+ z(1)) = g,\\
\lambda z + \tau^{-1}z_{\rho}=h.
\end{array}
\right.
\end{equation}
Suppose that we have found $u$ with the appropriate regularity. Then,
\begin{equation}\label{numerare}
v=\lambda u -f
\end{equation}
and we can determine $z.$ Indeed, by (\ref{domain}),
\begin{equation}\label{F3}
z(0)= B^*u,
\end{equation}
and, from (\ref{max}),
\begin{equation}\label{F4}
\lambda z(\rho ) + \tau^{-1} z_{\rho} (\rho )=h(\rho)\quad \mbox{\rm for } \ \rho\in (0,1).
\end{equation}
Then, by (\ref{F3}) and (\ref{F4}), we obtain
\be\label{defz}
z(\rho) = B^*u e^{-\lambda\rho\tau} + \tau e^{-\lambda\rho\tau} \int_0^{\rho} h(\sigma ) e^{\lambda\sigma\tau} d\sigma.
\ee
In particular, we have
\begin{equation}\label{F5tilde}
z(1)=   B^*u e^{-\lambda\tau} + z_0,
\end{equation}
with $z_0\in H_1$ defined by
\begin{equation}\label{F5star}
z_0=\tau e^{-\lambda\tau}\int_0^1 h(\sigma ) e^{\lambda\sigma\tau} d\sigma.
\end{equation}
This expression in (\ref{max}) shows that the function $u$ verifies formally
$$
\lambda^2 u+B(aB^*(\lambda u-f)+    B^*u e^{-\lambda\tau} +z_0)=g+\lambda f,$$
that is,
\begin{equation}\label{F6}
\lambda^2 u+ (\lambda a+   e^{-\lambda\tau}) BB^*u= g+\lambda f +B(aB^* f) - Bz_{0}.
\end{equation}
Problem (\ref{F6}) can be reformulated as
\begin{equation}\label{F7}
(\lambda^2 u+ (\lambda a+   e^{-\lambda\tau}) BB^*u, w)_H =
(g+\lambda f +B(aB^* f)-Bz_{0},w)_H,  \quad \forall\
w \in V.
\end{equation}
Using the definition of the adjoint of $B$, we get
\begin{equation}\label{F8}
\lambda^2  (u, w)_H +(\lambda a+   e^{-\lambda\tau}) (B^*u,  B^*w)_{H_1} =
(g+\lambda f,w)_H + (aB^*f-z_0, B^*w)_{H_1},
\ \forall\,
w\in V.
\end{equation}
As the left-hand side of (\ref{F8})
is coercive on ${ D}(B^*)$,
the Lax--Milgram lemma guarantees the existence and uniqueness of a solution $u\in V$ of (\ref{F8}).
Once $u$ is obtained we define $v$ by \rfb{numerare} that belongs to $V$
and $z$ by \rfb{defz} that belongs to $H^1(0,1; H_1)$.
Hence we can set $r=a B^*v+z(1)$, it belongs to $H_1$ but owing to \rfb{F8}, it fulfils
\[
\lambda (v,w)_H+(r, B^*w)_{H_1}= (g,w)_H, \ \forall w\in { D}(B^*),
\]
or equivalently
\[
(r, B^*w)_{H_1}= (g-\lambda v ,w)_H,\  \forall w\in { D}(B^*).
\]
As $g-\lambda v\in H$, this implies that $r$ belongs to ${ D}(B)$ with
\[
Br=g-\lambda v.
\]
This shows that the triple $U=(u,v,z)$ belongs to ${ D} ({\cal A})$ and
satisfies \rfb{surj}, hence $\lambda I - {\cal A}$ is surjective
for every $\lambda >0.$
\end{proof}

We have then the following result.
\begin{proposition}\label{propexistunic}
The system
\rfb{a1}--\rfb{a5} is well--posed. More precisely, for every
$(u_0,u_1,f_0) \in \cH$, there exists a unique solution $(u,v,z) \in C(0,+\infty, \cH)$ of \rfb{formulA}. Moreover, if $(u_0,u_1,f_0)\in D(\cA)$ then $(u,v,z) \in C(0,+\infty, D(\cA)) \cap C^1(0,+\infty, {\mathcal  H})$ with $v=u^\prime$ and $u$ is indeed a solution  of \rfb{a1}--\rfb{a5}.
\end{proposition}

\section{The spectral analysis \label{sectionspectre}}
\setcounter{equation}{0}
As ${ D}(B^*)$ is compactly embedded into $H$,   the operator $BB^* : { D}(BB^*) \subset H \rightarrow H$ has a compact resolvent. Hence let $(\lambda_k)_{k\in \nline^*}$
be  the set of eigenvalues of $BB^*$ repeated according to their multiplicity (that are positive real numbers
and are such that $\lambda_k\to +\infty$ as $k\to +\infty$)
and  denote by $(\varphi_k)_{k\in \nline^*}$ the corresponding eigenvectors
that form an orthonormal basis of $H$ (in particular for all   $k\in \nline^*$,
$BB^*\varphi_k=\lambda_k \varphi_k$).

\subsection{The discrete spectrum \label{ssectionspectrediscret}}
We have the following lemma.

\begin{lemma}\label{lspdiscret}
If $\tau\leq a$, then any eigenvalue $\lambda$ of $\cA$ satisfies
$\Re\lambda<0$.
\end{lemma}

\begin{proof}
Let $\lambda\in \mathbb C$ and  $U=(u,v,z)^\top\in D(\cA)$ be such that
 \[
(\lambda I- {\cal A}) \left (
\begin{array}{l}
u\\
v\\
z
\end{array}
\right )=0,
\]
or equivalently
\begin{equation}\label{eigenvalue}
\left\{
\begin{array}{l}
v = \lambda u,\\
- B (aB^* v+ z(\cdot,1))=\lambda v,\\
-\tau^{-1}z_{\rho}=\lambda z.
\end{array}
\right.
\end{equation}
By (\ref{F3}), we find that
\begin{equation}\label{F5eigen}
z(\rho )=\lambda^{-1} B^* v e^{-\lambda\rho\tau}.
\end{equation}
Using this property in (\ref{eigenvalue}), we find that
$u\in { D}(B^*)$ is solution of
\[
\lambda^2 u+(a \lambda + e^{-\lambda\tau}) BB^* u=0.
\]

Hence a non trivial solution exists if and only if
there exists $k\in {\mathbb N}^*$ such that
\begin{equation}\label{eigenvalue2}
\frac{\lambda^2}{a \lambda+    e^{-\lambda\tau}}=-\lambda_k.
\end{equation}

This condition implies that
$\lambda$ does not belong to

\begin{equation}\label{sigma}
\Sigma:=\{\lambda \in {\mathbb C}: a \lambda+     e^{-\lambda\tau}=0\},
\end{equation}
and that
\begin{equation}\label{eigenvalue3}
e^{-\lambda\tau}+\frac{\lambda^2}{\lambda_k}+a \lambda=0.
\end{equation}
Writing $\lambda=x+iy$, with $x, y\in \mathbb R$, we see that this identity is equivalent to
\begin{eqnarray}\label{eigenvalue4}
e^{-\tau x}\cos (\tau y)+\frac{x^2-y^2}{\lambda_k}+a x=0,
\\
-e^{-\tau x}\sin (\tau y)+\frac{2xy}{\lambda_k}+a y=0.
\label{eigenvalue5}
\end{eqnarray}

The second  equation is equivalent to
\[
e^{\tau x}\Big(\frac{2x }{\lambda_k}+a\Big)y=\sin (\tau y).
\]

Hence if $y\ne0$, we will get
\[
\frac{e^{\tau x}}{\tau}\Big(\frac{2x }{\lambda_k}+a\Big)=\frac{\sin (\tau y)}{\tau y}.
\]
As the modulus of the right-hand side is $\leq 1$, we obtain
\[
\Big|\frac{e^{\tau x}}{\tau}\Big(\frac{2x }{\lambda_k}+a\Big)\Big|\leq 1,
\]
or equivalently
\[
\Big|\frac{2x }{\lambda_k}+a \Big|\leq  \tau e^{-\tau x}.
\]

Therefore if $x\geq 0$, we find that
\[
 \frac{2x }{\lambda_k}+a  \leq  \tau e^{-\tau x}\leq \tau,
\]
which implies that
\[
 \frac{2x }{\lambda_k}   \leq    \tau-a.
\]
For $\tau<a$, we arrive to a contradiction.
For $\tau=a$,   the sole possibility is $x=0$
and by \rfb{eigenvalue5}, we find that
\[
\sin (\tau y)=\tau y,
\]
which yields $y=0$ and again we obtain a contradiction.

If $y=0$, we see that (\ref{eigenvalue5}) always holds
and (\ref{eigenvalue4}) is equivalent to
\[
e^{-\tau x} =-x(\frac{x }{\lambda_k}+a).
\]
This equation has no non--negative solutions $x$ since for $x\geq 0$,
the left hand side is positive while the right--hand side is non positive,
hence again if a solution $x$ exists, it has to be negative.

The proof of the lemma is complete.
\end{proof}

If $a<\tau$, we  now show that there exist some pairs of $(a,\tau)$ for which the system \rfb{a1}--\rfb{a5} becomes unstable.
Hence the condition $\tau \leq a$ is optimal for the stability of this  system.

\begin{lemma}\label{lspdiscretb}
There exist  pairs of  $(a,\tau)$ such that
$0<a<\tau$ and for which the associated operator $\mathcal A$ has a pure imaginary eigenvalue.
\end{lemma}
\begin{proof}
We look for a purely imaginary eigenvalue $iy$ of $\mathcal A$, hence system \rfb{eigenvalue4}--\rfb{eigenvalue5}
reduces to
\begin{eqnarray}\label{eigenvalue4x=0}
\cos (\tau y)=\frac{y^2}{\lambda_k},
\\
\sin (\tau y)=a y.
\label{eigenvalue5x=0}
\end{eqnarray}
Such a solution exists if
\begin{equation}\label{charaeq}
\frac{y^4}{\lambda_k^{2}}+ a^{2}y^2=1.
\end{equation}
One  solution of this equation is
\[
y_k=\left(\frac{-a^{2}\lambda_k^{2}+\sqrt{a^{4}\lambda_k^{4}+4\lambda_k^{2}}}{2}\right)^{\frac{1}{2}}.
\]
We now take any $\tau\in (0,\frac{\pi}{2y_k})$ and
$a=\frac{\sin (\tau y_k)}{y_k}$.
Then \rfb{eigenvalue5x=0} automatically holds, while \rfb{eigenvalue4x=0} is valid owing to \rfb{charaeq} (as $\cos (\tau y_k)>0$).
Finally $a<\tau$ because
\[
\frac{a}{\tau}=\frac{\sin (\tau y_k)}{\tau y_k}<1.
\]
Therefore with such a choice of $a$ and $\tau$, the operator $\mathcal A$ has a purely imaginary eigenvalue equal to $iy_k$.
\end{proof}

\subsection{The continuous spectrum \label{ssectionspectrecontinu}}

Inspired from section 3 of \cite{ANpart1:14}, by using a Fredholm alternative technique, we perform the spectral analysis of the operator $\cA$.

Recall that an operator $T$ from a Hilbert space $X$ into itself is called singular if there exists a sequence
$u_n\in D(T)$ with no convergent subsequence such that
$\|u_n\|_X=1$ and $Tu_n\to 0$ in $X$, see \cite{wolf:59}.
According to Theorem 1.14 of \cite{wolf:59} $T$ is   singular if and only if its kernel is infinite dimensional or its range is not closed.
Let $\Sigma$ be the set defined in (\ref{sigma}).
The following results hold:

\begin{theorem}\label{thmbeale}
\begin{enumerate}
\item If $\lambda\in \Sigma$, then $\lambda I-\cA$ is singular.
\item If $\lambda\not\in \Sigma$, then $\lambda I-\cA$ is a Fredholm operator of index zero.
\end{enumerate}
\end{theorem}
\begin{proof}
For the proof of point 1, let us fix $\lambda\in \Sigma$
and for all $k\in \nline^*$ set
\[
U_k=(u_k, \lambda u_k, B^*u_k e^{-\lambda \tau\cdot})^\top,
\]
with $u_k=\frac{1}{\sqrt{\lambda_k}} \varphi_k$. Then $U_k$ belongs to ${ D}({\cal A})$
and easy calculations yield (due to the assumption $\lambda\in \Sigma$)
\[
(\lambda I-{\cal A})U_k=
\lambda^2 (0,u_k, 0)^\top.
\]
Therefore we deduce that
\[
\|(\lambda I-{\cal A})U_k\|_{\h}\to 0, \hbox{ as } k\to\infty.
\]
Moreover due to the property $\| B^*u_k \|_{H_1}=1$, there exist positive constants  $c, C,$ such that
\[
c\leq \|U_k\|_{\h}\leq C, \forall\ k\in \nline^*.
\]

This shows that $\lambda I-{\cal A}$ is singular.

For all $\lambda \in {\mathbb C}$, introduce the (linear and continuous) mapping
$A_\lambda$ from $V$ into its dual by
\[
\langle A_\lambda v, w\rangle _{V'-V} =
\lambda^2 (v,w)_H + (a\lambda + e^{-\lambda\tau})  (B^* v,  B^* w)_{H_1}, \
\forall\  v,w\in { D}(B^*).
\]
Then from the proof of Lemma \ref{lmaxdiss}, we know that for $\lambda>0$,
$A_\lambda$ is an isomorphism.

Now for $\lambda \in {\mathbb C}\setminus \Sigma$, we can introduce the operator
\[
B_\lambda=(a\lambda + e^{-\lambda\tau})^{-1}A_\lambda.
\]
Hence for $\lambda \in   {\mathbb C}\setminus \Sigma$, $A_\lambda$ is a Fredholm operator of index 0
if and only if $B_\lambda$ is a Fredholm operator of index 0.
Furthermore for $\lambda,\mu \in {\mathbb C}\setminus \Sigma$
as
$B_\lambda-B_\mu$ is a multiple of the identity operator, due to the compact embedding of $V$ into $V'$,
and as $B_\mu$ is an isomorphism for $\mu>0$, we finally deduce that
$A_\lambda$ is a Fredholm operator of index 0 for all  $\lambda \in   {\mathbb C}\setminus \Sigma$.

Now we  readily check  that, for any $\lambda\in \C\setminus \Sigma$, we have the equivalence

\begin{equation}\label{serge14/08:2}
u\in \ker A_\lambda \iff (u,\lambda u, B^*u e^{-\lambda \tau \cdot})^\top \in \ker (\lambda I-\A).\end{equation}

This equivalence implies that

\be\label{kernel}
\hbox{ dim } \ker (\lambda I-\A)=\hbox{ dim } \ker A_\lambda,\  \forall\  \lambda\in \C\setminus \Sigma.
\ee

For the range property for all $\lambda\in \C\setminus \Sigma$ introduce the inner product
$$
(u,z)_{\lambda, V}:=\Big((u,\lambda u, B^*u e^{-\lambda \tau \cdot})^\top,
(z,\lambda z, B^*z e^{-\lambda \tau \cdot})^\top\big)_\h,
$$
on $V$ whose associated norm is equivalent to the standard one.

Denote by $\{y^{(i)}\}_{i=1}^{N}$ an orthonormal basis of $\ker A_\lambda$ for this new inner product
(for shortness the dependence of $\lambda$ is dropped), i.e.%, such that
$$
(y^{(i)},y^{(j)})_{\lambda, V}=\delta_{ij}, \forall\  i,j=1, \ldots, {N}.
$$
Finally, for all $i=1, \ldots, {N}$,  we set
$$
Z^{(i)}=(y^{(i)},\lambda y^{(i)}, B^*y^{(i)}e^{-\lambda\tau\cdot })^\top,
$$
the element of $\ker (\lambda I-\A)$ associated with $y^{(i)}$ that are orthonormal with respect to the inner product of $\h$.

Let us now show that
for all $\lambda\in \C\setminus \Sigma$, the range $R(\lambda I-\A)$ of $\lambda I-\A$ is closed.
Indeed, let us consider a sequence   $U_n=(u_n,v_n,z_n)^\top\in { D}(\A)$ such that
\begin{equation}\label{serge14/08:7}
(\lambda I-\A)U_n=F_n=(f_{n},g_{n},h_{n})^\top\to F=(f,g,h)^\top
\hbox{ in } \h.
\end{equation}
Without loss of generality we can assume that
\begin{equation}\label{serge14/08:5}
(U_n, Z^{(i)})_\h=-\alpha_{n, i}, \ \forall\  i=1, \ldots, {N},
\end{equation}
where
$$
\alpha_{n, i}:=((0, f_{n}, -\tau e^{-\lambda\tau\cdot }\int_0^{\cdot} h_n(\sigma) e^{\lambda\sigma\tau } d\sigma )^\top, Z^{(i)})_\h.
$$
Indeed, if this is not the case, we can consider
$$\tilde U_n=U_n-\sum_{i=1}^{N} \beta_i Z^{(i)}$$
that still belongs to $D(\A)$ and satisfies
$$
(\lambda I-\A)\tilde U_n=F_n,
$$
as well as
$$
(\tilde U_n, Z^{(i)})_\h=-\alpha_{n, i}, \ \forall\  i=1, \ldots, {N},
$$
 by setting
$$\beta_i= (U_n, Z^{(i)})_\h+\alpha_{n, i}, \  \forall\  i=1, \ldots, {N}.
$$

Note that the condition (\ref{serge14/08:5}) is equivalent to
$$
(u_n, y^{(i)})_{\lambda,V}=0,\  \forall\ i=1, \ldots, {N}.
$$
In other words,
\begin{equation}\label{serge14/08:6}
u_n\in (\ker A_\lambda)^{\perp_{\lambda,V}},
\end{equation}
where ${ }^{\perp_{\lambda,V}}$ means that the orthogonality is taken with respect to the inner product
$(\cdot,\cdot)_{\lambda, V}$.

Returning  to (\ref{serge14/08:7}), the arguments of the proof of
Lemma \ref{lmaxdiss} imply that
$$
A_\lambda u_n= L_{F_n} \hbox{ in } V',
$$
where $L_F$ is defined by
\[
L_F(w):=(g,w)_H -
\tau e^{-\lambda\tau}(\int_0^1 h(\sigma ) e^{\lambda\sigma\tau} d\sigma,
B^*w)_{H_1}+(\lambda f+a B^*f, w)_{H_1},
\]
when $F=(f,g,h)^\top$. But it is easy to check that
$$
L_{F_n} \to L_F\hbox{ in } V'.
$$
Moreover, as $\lambda\in \C\setminus \Sigma$, $A_\lambda$
is an isomorphism from
$(\ker A_\lambda)^{\perp_{\lambda,V}}$ into $R(A_\lambda)$, hence by (\ref{serge14/08:6}) we deduce that there exists a positive constant $C(\lambda)$ such that
$$
\|u_n-u_m\|_{V}\leq C(\lambda) \|L_{F_n} -L_{F_m}\|_{V'}, \forall\  n,m \in \nline.
$$
Hence, $(u_n)_n$ is a Cauchy sequence in $V$, and therefore there exists $u\in V$ such that
$$
u_n\to u \ \hbox{ in } V,
$$
as well as
$$
A_\lambda u= L_{F}\ \hbox{ in } V'.
$$
Then  defining  $v$ by \rfb{numerare}
and $z$ by \rfb{defz}, we deduce that
$U:=(u,v,z)^\top$ belongs to ${ D}(\A)$ and
$$
(\lambda I-\A)U=F.
$$
In other words, $F$ belongs to $R(\lambda I-\A)$. The closedness of $R(\lambda I-\A)$ is thus proved.

At this stage, for any $\lambda\in \C\setminus \Sigma$, we show that
\begin{equation}\label{serge14/08:3}
 \hbox{ codim } R(A_\lambda)=\hbox{ codim } R(\lambda I-\A),\end{equation}
where for $W\subset \h$, codim $W$ is the dimension of the orthogonal in $\h$ of $W$,
while for $W'\subset V'$, codim $W'$ is the dimension of the annihilator
\[
A:=\{v\in V:
\langle v, w\rangle_{V-V'}=0,\ \forall\ w\in W'\},
\]
 of $W'$ in $V$.

Indeed, let us set $N= \hbox{ codim } R(A_\lambda)$, then there exist
$N$ elements $\varphi_i\in V,$  $i=1,\ldots, N,$ such that
$$
f\in R(A_\lambda) \iff f\in V' \hbox{ and } \langle f,\varphi_i\rangle_{V'-V}
= 0,\ \forall\  i=1,\ldots, N.
$$
Consequently, for $F\in \h$, if $L_F$ (that belongs to $V'$) satisfies
\begin{equation}\label{serge14/08:10}
L_F(\varphi_i)= 0,\ \forall\  i=1,\ldots, N,
\end{equation}
then there exists a solution $u\in V$ of
$$
A_\lambda u=L_F \hbox{ in V'},
$$
and   the arguments of the proof of
Lemma \ref{lmaxdiss} imply that $F$ is in $R(\lambda I-\A)$. Hence, the $N$ conditions on $F\in \h$ from
(\ref{serge14/08:10}) allow to show that it belongs to $R(\lambda I-\A)$,
and therefore
\begin{equation}\label{serge14/08:12}
\hbox{ codim } R(\lambda I-\A)\leq N=\hbox{ codim } R(A_\lambda).
\end{equation}
This shows that $\lambda I-\A$ is a Fredholm operator.

Conversely, set $M= \hbox{ codim } R(\lambda I-\A)$, then there exist $M$ elements
$\Psi_i=(u_i,v_i, z_i)\in \h,$  $i=1,\ldots, M,$ such that
$$
F\in R(\lambda I-\A) \iff F\in \h \hbox{ and } (F,\Psi_i)_{\h}= 0,\forall\  i=1,\ldots, M.
$$
Then, for any $g\in H$,
if
\begin{equation}\label{serge14/08:11}
(g,v_i)_{H}=((0,g,0)^\top,\Psi_i)_{\h}= 0,\forall\  i=1,\ldots, M,
\end{equation}
there exists $U=(u,v,z)^\top\in { D}(\A)$ such that
$$
(\lambda I-\A)U=(0,g,0),
$$
which implies that
\[
A_\lambda u=g.
\]
This shows that
$$
R(A_\lambda)\supset H_0,$$
where $H_0:=
\{g\in H \hbox{ satisfying }(\ref{serge14/08:11})\}$.
This inclusion implies that (here $\perp$ means the annihilator of the set in $V$)
\[
R(A_\lambda)^\perp\subset H_0^\perp.
\]
Therefore
\beqs
R(A_\lambda)^\perp
&\subset& \{v\in V : \langle v, g\rangle_{V-V'}=0, \forall\ g\in H_0\}
\\
&=&\{v\in V : ( v, g)_H=0, \forall\ g\in H_0\}
\\
&\subset& \hbox{ Span }\{v_i\}_{i=1}^M\cap V.
\eeqs
Hence,
\begin{equation}\label{serge14/08:13}
\hbox{ codim } R(A_\lambda)\leq M=\hbox{ codim } R(\lambda I-\A).
\end{equation}

The inequalities (\ref{serge14/08:12}) and (\ref{serge14/08:13}) imply (\ref{serge14/08:3}).
 \end{proof}
\begin{lemma}\label{lsigma}
If $\tau\leq a$, then
\[
\Sigma\subset \{\lambda \in {\mathbb C}:\Re\lambda<0\}.\]
\end{lemma}
 \begin{proof}
Let  $\lambda=x+iy  \in\Sigma$, with $x, y\in \mathbb R$ we deduce that
\begin{eqnarray*}
ax+e^{-\tau x}\cos (\tau y)=0,
\\
a y-e^{-\tau x}\sin (\tau y)=0.
\end{eqnarray*}
This corresponds to the system (\ref{eigenvalue4})--(\ref{eigenvalue5}) with $k=\infty$,
hence the arguments as in the proof of Lemma \ref{lspdiscret} yield the result.
\end{proof}

\begin{corollary} \label{locspec}
It holds
\[
\sigma(\cA)=\sigma_{pp}(\cA)\cup  \Sigma,
\]
and
therefore if $\tau\leq a$
\[
\sigma(\cA)\subset \{\lambda \in {\mathbb C}:\Re\lambda<0\}.\]
\end{corollary}
\begin{proof}
By Theorem \ref{thmbeale},
\[
{\mathbb C}\setminus \Sigma\subset \sigma_{pp}(\cA)\cup \rho(\cA).
\]
The first assertion directly follows.

The second assertion follows from Lemmas \ref{lspdiscret}
and \ref{lsigma}.
\end{proof}

\section{Asymptotic behavior \label{sectasanalysis}}
In this section, we show that if  $\tau\leq a$ and $\xi > \frac{2\tau}{a}$, the semigroup $e^{t\cA}$ decays to the null steady state with an exponential decay rate. To obtain this, our technique is based on a frequency domain approach and combines a contradiction argument to carry out a special analysis of the resolvent.

\begin{theorem} \label{lr}
If $\xi > \frac{2\tau}{a}$ and $\tau\leq a$, then there exist  constants $C, \omega > 0$ such that the semigroup $e^{t\cA}$ satisfies the following estimate
\be
\label{estexp}
\left\| e^{t \cA}\right\|_{\cL(\cH)} \leq C \, e^{-\omega t}, \ \forall \ t > 0.
\ee
\end{theorem}

\begin{proof} [Proof of theorem  \ref{lr}]
We will employ the following frequency domain theorem for uniform stability from
\cite[Thm 8.1.4]{JacobZwart} of a $C_0$ semigroup on a Hilbert space:

\begin{lemma}
\label{pruss}
A $C_0$ semigroup $e^{t{\cal L}}$ on a Hilbert space ${\cal H}$ satisfies
$$||e^{t{\cal L}}||_{{\cal L}({\cal H})} \leq C \, e^{-\omega t},$$
for some constant $C >0$ and for $\omega>0$ if and only if
\be
\Re \lambda < 0, \  \forall \, \lambda \in \sigma (\cL), \label{1.8w} \ee
and \be \sup_{\Re \lambda \geq 0}  \|(\lambda I -{\cal L})^{-1}\|_{{\cal L}({\cal H})} <\infty. \label{1.9}
\ee
where $\sigma({\cal L})$ denotes the spectrum of the operator
${\cal L}$.
\end{lemma}
According to Corollary \ref{locspec} the spectrum of $\cA$ is fully included into $\Re \lambda <0$,
which clearly implies \rfb{1.8w}.
Then the proof of Theorem \ref{lr} is based on the following lemma
that shows that (\ref{1.9}) holds with $\mathcal{L}=\cA$.

\begin{lemma}\label{lemresolvent}
The resolvent operator of $\cA$ satisfies condition
\be \sup_{\Re \lambda \geq 0}  \|(\lambda I -{\cal A})^{-1}\|_{{\cal L}({\cal H})} <\infty. \label{1.9bis}
\ee
\end{lemma}

\begin{proof}
Suppose that condition \eqref{1.9bis} is false.
By the Banach-Steinhaus Theorem (see \cite{brezis}), there exists a sequence of complex numbers $\lambda_n$ such that $\Re \lambda_n \geq 0,\  |\lambda_n| \rightarrow +\infty$ and a sequence of vectors
$Z_n= (u_{n},v_{n},z_n )^t\in D(\cA)$ with
\be
\label{cont}
\|Z_n\|_{\cH} = 1
\ee
such that
\be
|| (\lambda_n I - \cA)Z_n||_{\cH} \rightarrow 0\;\;\;\; \mbox{as}\;\;\;n\rightarrow \infty,
\label{1.12} \ee
i.e.,
\be \lambda_n u_n - v_{n} \equiv f_{n}\rightarrow 0 \;\;\; \mbox{in}\;\; { D}(B^*),
\label{1.13}\ee
 \be
   \lambda_n
v_{n} + a \,B(B^* v_n +  z_{n} (1))  \equiv g_{n} \rightarrow 0 \;\;\;
\mbox{in}\;\; H,
\label{1.13b} \ee
\be \label{1.14}
 \lambda_n \, z_n + \tau^{-1} \partial_\rho z_n \equiv h_n \rightarrow 0 \; \; \; \mbox{in} \;\; L^2((0,1);H_1).
\ee

Our goal is to derive from \eqref{1.12} that $||Z_n||_{\cH}$ converges to zero, that furnishes a contradiction.

We notice that from \rfb{dissp} and \rfb{1.13} we have
\beqs
|| (\lambda_n I - {\cal A})Z_n||_{{\cal H}} \geq |\Re \left((
\lambda_n I - {\cal A})Z_n, Z_n\right)_{{\cal H}} |
\\
\geq
\Re \, \lambda_n - a_*^{-1} \left\|B^*u_{n} \right\|^2_{H_1} +
\left( \frac{\xi}{2\tau} - \frac{1}{a} \right) \, \left\|z_n(1)\right\|^2_{H_1} + \frac{a}{2} \, \left\|B^*v_{n}\right\|^2_{H_1} \\
=
\Re \, \lambda_n - a_*^{-1} \left\|\frac{B^*v_{n} + B^*f_{n}}{\lambda_n}\right\|^2_{H_1} +
\left( \frac{\xi}{2\tau} - \frac{1}{a} \right) \, \left\|z_n(1)\right\|^2_{H_1} + \frac{a}{2} \, \left\|B^*v_{n}\right\|^2_{H_1}.
\eeqs
Hence using the inequality
\[
\left\|B^*v_{n} + B^*f_{n} \right\|^2_{H_1} \leq 2 \|B^*v_{n}\|^2_{H_1} + 2\| B^*f_{n} \|^2_{H_1},
\]
 we obtain that
\beqs
|| (\lambda_n I - {\cal A})Z_n||_{{\cal H}} \geq
\Re \, \lambda_n-2 a_*^{-1} |\lambda_n|^{-2}\left\| B^*f_{n}\right\|^2_{H_1} +
\left( \frac{\xi}{2\tau} - \frac{1}{a} \right) \, \left\|z_n(1)\right\|^2_{H_1}
\\
 + (\frac{a}{2} -2 a_*^{-1} |\lambda_n|^{-2})
\, \left\|B^*v_{n}\right\|^2_{H_1}.
\eeqs
Hence for $n$ large enough, say $n\geq n^*$, we can suppose that
\[
\frac{a}{2} -2 a_*^{-1} |\lambda_n|^{-2}\geq \frac{a}{4}.
\]
and therefore for all $n\geq n^*$, we get
\beqs
|| (\lambda_n I - {\cal A})Z_n||_{{\cal H}} \geq
\Re \, \lambda_n-2 a_*^{-1} |\lambda_n|^{-2}\left\| B^*f_{n}\right\|^2_{H_1} +
\left( \frac{\xi}{2\tau} - \frac{1}{a} \right) \, \left\|z_n(1)\right\|^2_{H_1} \\+ \frac{a}{4}
\, \left\|B^*v_{n}\right\|^2_{H_1}.
\label{1.15}
\eeqs
By this estimate, \rfb{1.12} and \rfb{1.13}, we deduce that
\be
\label{cvk}
z_{n}(1) \rightarrow 0, \ B^*v_{n} \rightarrow 0,\  \hbox{ in } H_1,
\hbox{ as } n\to\infty ,
\ee
and in particular, from the coercivity (\ref{cors}),  that
$$
v_{n} \rightarrow 0, \hbox{ in } H, \hbox{ as } n\to\infty.
$$
This implies according to \rfb{1.13} that
\be
\label{c4}
u_n = \frac{1}{\lambda_n} v_n + \frac{1}{\lambda_n} \, f_n  \rightarrow 0, \;\;\;
\mbox{ in }\;\; { D}(B^*),  \hbox{ as } n\to\infty,
\ee
as well as
\be
\label{c5}
z_n(0) = B^*u_{n} \rightarrow 0, \;\;\;
\mbox{in}\;\; H_1,  \hbox{ as } n\to\infty.
\ee
By integration of the identity \rfb{1.14}, we have
\be
\label{c7}
z_n(\rho) = z_n (0) \, e^{-\tau \lambda_n \rho} + \tau \, \int_0^\rho e^{-\tau \lambda_n (\rho - \gamma)} \, h_n(\gamma) \, d \gamma.
\ee
Hence recalling that $\Re\lambda_n\geq0$
\[
\int_0^1 \|z_n(\rho) \|^2_{H_1}\,d\rho\leq
2 \|z_n (0)\|_{H_1}^2
+2\tau^2 \int_0^1 \int_0^\rho \|h_n(\gamma)  \|^2_{H_1} \, d \gamma \rho\,d\rho\to 0, \hbox{ as } n\to\infty.
\]
All together we have shown that $\|Z_n\|_{\cH}$ converges to zero, that clearly contradicts $\left\|Z_n\right\|_{\cH}=1$.
\end{proof}

The two hypotheses of Lemma \ref{pruss} are proved, then \rfb{estexp} holds. The proof of Theorem \ref{lr} is then finished.
\end{proof}

\section{Application to the stabilization of the wave equation with delay and a Kelvin--Voigt damping \label{sectionmulti}}
\setcounter{equation}{0}
We study the internal stabilization of a  delayed wave equation. More precisely, we consider the
system  given by :
\begin{align}
{}&u_{tt}(x,t) - a \, \Delta u_{t}(x,t)  - \Delta u(x,t - \tau) = 0 ,  && \mbox{\rm in} \quad \Omega \times (0, + \infty),\label{a1ex}\\
{}&u = 0,&& \mbox{\rm on}\quad \partial \Omega \times (0,+\infty), \label{a2ex}\\
{}&u(x,0) = u_0(x),  \quad  u_t(x,0)=u_1(x), && \mbox{\rm in} \quad \Omega,\label{a4ex}
\\
{}&\nabla u(x,t-\tau) = f_0(t-\tau),  \quad  && \mbox{\rm in} \quad \Omega \times (0,\tau) ,\label{a5ex}
\end{align}
where $\Omega$ is a smooth
open bounded domain of $\rline^n$ and $a, \tau>0$ are constants.

This problem enters in our abstract framework with
 $H = L^2(\Omega), B = -$ div $:  { D}(B) = H^1(\Omega)^n \rightarrow L^2(\Omega), B^* = \nabla : { D}(B^*) = H^1_0(\Omega) \rightarrow H_1:=L^2(\Omega)^n$,   the assumption \rfb{cors} being satisfied
owing to Poincar\'e's inequality.
The operator ${\cal A}$ is then given by
$$
{\cal A} \left (
\begin{array}{l}
u\\
v\\
z
\end{array}
\right ):=
\left (
\begin{array}{l}
v\\
a \Delta v + \hbox{ div } z(\cdot,1)\\
-\tau^{-1}z_{\rho}
\end{array}
\right ),
$$
\noindent with domain
\begin{equation}\label{domainex}
\begin{array}{l}
\displaystyle{
D ({\cal A}):=\Big \{\ (u,v,z)^\top\in
H^1_0(\Omega)
\times H^1_0(\Omega ) \times L^2(\Omega; H^1(0,1))\ : a \nabla v + z(\cdot, 1)\in H^1(\Omega),}\\
\hspace{2.5cm} \displaystyle{\nabla u=z(\cdot,0)\ \mbox{\rm in}\ \Omega
 \Big\},}
\end{array}
\end{equation}
in the Hilbert space
\begin{equation}\label{spaceex}
{\cal H}:= H^1_0 (\Omega)\times L^2 (\Omega)\times L^2(\Omega\times (0,1)).
\end{equation}

According to Lemma \ref{locspec} and Theorem \ref{lr} we have:
\begin{corollary}
If $\tau \le a$, the system \rfb{a1ex}--\rfb{a5ex} is exponentially stable in ${\cal H}$, namely
for $\xi > \frac{2\tau}{a}$, the energy
\[
E(t)=\frac12\left(
\int_\Omega(|\nabla u(x,t)|^2+|u_t(x,t)|^2)\,dx+
\xi \int_\Omega\int_0^1 |\nabla u(x,t-\tau\rho)|^2\,dx d\rho\right),
\]
satisfies
\[
E(t)\leq Me^{-\omega t} E(0), \ \forall\ t>0,
\]
for some positive constants $M$ and $\omega$.
\end{corollary}

\section*{Conclusion}
By a careful spectral analysis combined with a frequency domain approach, we have shown that the system \rfb{a1}--\rfb{a5} is exponentially stable if $\tau \leq a$
and that this condition is optimal.
But from the general form of \rfb{a1}, we can only consider interior Kelvin-Voigt dampings.
Hence an interesting perspective is to consider the wave  equation with dynamical Ventcel
boundary conditions with a delayed term and a Kelvin-Voigt damping.

%and the same thing where the operator $B$ acting on the boundary (i.e., unbounded). These questions will be treated in the future.


\begin{thebibliography}{99}

\bibitem{ANpart1:14}
{\sc Z.~Abbas and S.~Nicaise},
\newblock  {\em  The multidimensional wave equation with generalized acoustic boundary
  conditions {I}: Strong stability},
\newblock  SIAM J. Control Opt.,  2015,
\newblock to appear.



\bibitem{amman} {\sc E.~M.~Ait Ben Hassi, K.~Ammari, S.~Boulite and L.~Maniar}, {\em Feedback stabilization of a class of evolution equations with delay}, J. Evol. Equ., {\bf 1} (2009), 103-121.

\bibitem{ammarinicaise} {\sc K.~Ammari and S.~Nicaise}, {\em Stabilization of elastic systems by collocated feedback,} Lecture Notes in Mathematics, Vol. 2124, Springer-Verlag, 2015.

\bibitem{ANP1} {\sc K.~Ammari, S.~Nicaise and C.~Pignotti,} {\em Stabilization by switching time-delay,} Asymptotic Analysis, {\bf 83} (2013), 263--283.

\bibitem{ANP} {\sc K.~Ammari, S.~Nicaise and C.~Pignotti,}
{\em Feedback boundary stabilization of wave equations with interior delay,} Systems Control Lett., {\bf 59} (2010), 623--628.

\bibitem{at} {\sc K.~Ammari and M.~Tucsnak}, {\em Stabilization of second order evolution equations by a class of
unbounded feedbacks,} ESAIM Control Optim. Calc. Var., {\bf 6} (2001), 361-386.


\bibitem{batkaibook}
{\sc  A. B{\'a}tkai and S. Piazzera},
{\em Semigroups for delay equations,} Research Notes in Mathematics {\bf 10}, A. K. Peters, Wellesley MA (2005).

\bibitem{brezis} {\sc H.~Brezis,} {\em Analyse Fonctionnelle,} Th\'eorie et Applications, Masson, Paris, 1983.



\bibitem{Datko} {\sc R.~Datko,} {\em Not all feedback stabilized hyperbolic systems are robust with respect to small time delays in their feedbacks,} SIAM J. Control Optim., {\bf 26} (1988), 697-713.

\bibitem{DLP} {\sc R.~Datko, J.~Lagnese and P.~Polis,} {\em An exemple on the effect of time delays in boundary feedback stabilization of wave equations,} SIAM J. Control Optim., {\bf 24} (1985), 152-156.

\bibitem{Datko97} {\sc R.~Datko,}
{\em Two examples of ill-posedness with respect to time delays revisited,}
IEEE Trans. Automatic Control, {\bf 42} (1997), 511--515.

%\bibitem{pruss} {\sc J.~Pr\"{u}ss,}  On the spectrum of $C_0$-semigroups, {\em Trans. Amer. Math. Soc.}, {\bf 248}(1984),  847-857.

\bibitem{dreheretall}
{\sc  M.~Dreher, R.~Quintanilla and  R.~Racke},
{\em Ill-posed problems in thermomechanics}, Appl. Math. Letters,
 {\bf  22} (2009), 1374-1379.

\bibitem{Gugat}
{\sc M.~Gugat},
{\em Boundary feedback stabilization by time delay for one-dimensional wave equations}, IMA J. Math. Control Inform., {\bf  27} (2010), 189--203.

\bibitem{huang} {\sc F.~Huang,}  {\em Characteristic conditions for exponential
stability of linear dynamical systems in Hilbert space}, { Ann.
Differential Equations}, {\bf 1} (1985), 43-56.

\bibitem{JacobZwart}  {\sc B.~Jacob and H. ~Zwart,}
{\em  Linear Port-Hamiltonian Systems on Infinite-dimensional Spaces},
Operator Theory: Advances and Applications, {\bf 223}, Birkhauser,
2012.

\bibitem{jordanetall}
 {\sc P.M.  ~Jordan, W.~Dai and   R.E. ~Mickens},
{\em  A note on the delayed heat equation: Instability with respect
to initial data},  Mech. Res. Comm.,  {\bf 35} (2008), 414-420.

\bibitem{NPSicon}
{\sc S.~Nicaise and C.~Pignotti,} {\em Stability and instability results of the wave equation with a delay term in the boundary
or internal feedbacks}, {  SIAM J. Control Optim.}, {\bf 45} (2006), 1561--1585.

\bibitem{NVCOCV10} {\sc S.~Nicaise and J.~Valein,} {\em Stabilization of second order evolution equations with unbounded feedback with delay,} ESAIM Control Optim. Calc. Var., {\bf 16} (2010), 420--456.


\bibitem{pruss:93}
     {\sc  J.~Pr\"uss}, {\em Evolutionary integral equations  and applications},
Monograhs Mathematics, {\bf 87},
  {Birkh\"auser Verlag},  {Basel},  {1993}.

\bibitem{racke} {\sc R.~Racke,} {\em Instability of coupled systems with delay,} Commun. Pure Appl. Anal., {\bf 11} (2012), 1753--1773.

\bibitem{wolf:59}
{\sc F.~Wolf,} {\em On the essential spectrum of partial differential boundary problems,} Comm. Pure Appl. Math., {\bf 12} (1959), 211--228.

\end{thebibliography}
\end{document}